\documentclass{amsart}
\usepackage{amssymb,amsmath,latexsym,amscd,amsfonts}
\usepackage{graphicx}
\usepackage{epstopdf}
\usepackage{multirow}

\theoremstyle{plain}
\newtheorem{thm}{Theorem}
\newtheorem{lemma}[thm]{Lemma}

\newtheorem{cor}[thm]{Corollary}

\theoremstyle{definition}
\newtheorem{example}[thm]{Example}
\newtheorem{defn}[thm]{Definition}

\newcommand{\of}[1]{\!\left({#1}\right)}
\newcommand{\e}{\varepsilon}
\newcommand{\de}{\delta}

\newcommand{\Z}{\mathbb{Z}}

\newcommand{\oddw}{\ensuremath{w_o}}
\newcommand{\OWP}{\ensuremath{W_K(t)}}
\newcommand{\B}[2]{{#1}--{#2}}

\begin{document}

\markboth{Alissa S. Crans, Sandy Ganzell, Blake Mellor}
{The Forbidden Number of a Knot}

\title{The Forbidden Number of a Knot}

\author{ALISSA S. CRANS}
\address{Department of Mathematics, Loyola Marymount University, 1 LMU Drive, Los Angeles, California 90045}
\email{acrans@lmu.edu}

\author{SANDY GANZELL}
\address{Department of Mathematics and Computer Science, St. Mary's College of Maryland, St. Mary's City, MD 20686}
\email{sganzell@smcm.edu}

\author{BLAKE MELLOR}
\address{Department of Mathematics, Loyola Marymount University, 1 LMU Drive, Los Angeles, California 90045}
\email{blake.mellor@lmu.edu}

\maketitle

\begin{abstract}
Every classical or virtual knot is equivalent to the unknot via a sequence of extended Reidemeister moves and  the so-called forbidden moves. The minimum number of forbidden moves necessary to unknot a given knot is an invariant we call the {\it forbidden number}. We relate the forbidden number to several known invariants, and calculate bounds for some classes of virtual knots.
\end{abstract}
\section{Introduction}\label{intro}

The theory of virtual knots was introduced by Kauffman \cite{KauffmanVKT} as a generalization of classical knot theory, motivated in part by the desire to provide a bijective correspondence between knots and {\em Gauss diagrams}.  Like the classical theory, virtual knot theory has a useful diagrammatic approach since virtual diagrams can be thought of as 4-valent graphs with extra structure at the vertices. In the classical theory, this extra structure is indicated by overcrossings and undercrossings, whereas in the virtual theory, a third kind of crossing is allowed, namely virtual crossings, which are indicated in the diagrams with a small circle around the vertex.

Equivalence of virtual knots may be defined by means of a set of local moves (the extended Reidemeister moves) on their diagrams.\footnote{Equivalence may also be viewed as isotopy in thickened surfaces; see, for example, \cite{KauffmanINTRO}, \cite{CKS} or \cite{Kuperberg}.} Figure \ref{VRmoves} illustrates the extended Reidemeister moves: the classical Reidemeister moves $R1$, $R2$, $R3$, the virtual moves $V1$, $V2$, $V3$, and the semivirtual move $SV$.  We may then define a virtual knot to be an equivalence class of virtual diagrams modulo these moves.
\begin{figure}[htbp]
\begin{center}
\includegraphics{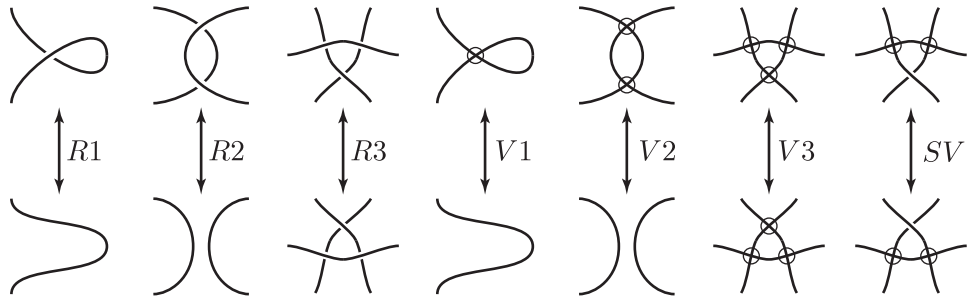}
\end{center}
\caption{The extended Reidemeister moves.}
\label{VRmoves}
\end{figure}
 
We remind the reader that it is important to restrict ourselves to these moves, and not consider physical movements in space, as our intuition can sometimes lead us astray with diagrams that contain virtual crossings. 
\begin{figure}[htbp]
\begin{center}
\includegraphics{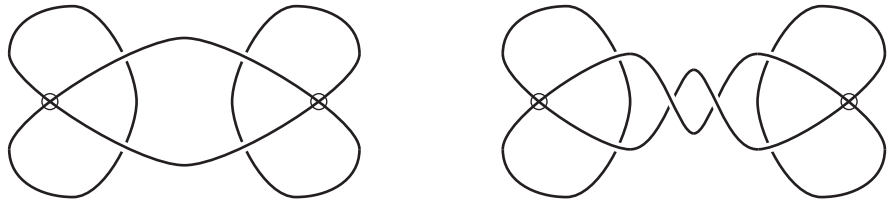}
\end{center}
\caption{Distinct virtual knots.}
\label{Kishino}
\end{figure}
For example, the virtual knots in Fig. ~\ref{Kishino} are distinct, even though they appear to differ only by a physical twist. These knots are distinguished by the arrow polynomial defined by Dye and Kauffman \cite{DyKa}.

Given an oriented (virtual) knot, we recall that its Gauss diagram is defined as follows:  First label all the classical crossings of the knot.  Next, traverse the knot, writing down the sequence of crossing labels (so each label appears twice); write this sequence around a circle.  Then add a chord to the circle for each crossing; each chord is drawn as an arrow directed from the label corresponding to going {\em over} a crossing to the label corresponding to going {\em under} the same crossing.  Finally, each arrow is also labeled with the sign of the crossing.
\begin{figure}[htbp]
\begin{center}
{\includegraphics{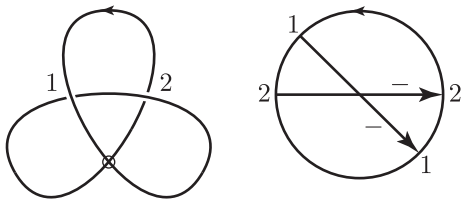}}
\end{center}
\caption{Gauss diagram of a virtual trefoil.}
\label{Gauss1}
\end{figure}
For example, Fig. ~\ref{Gauss1} shows the left-handed virtual trefoil knot and its corresponding Gauss diagram.  

The Reidemeister moves can be reinterpreted as moves on Gauss diagrams; the virtual moves $V1$, $V2$, $V3$ and $SV$ have no effect on the Gauss diagram.  While not every Gauss diagram can be realized as the diagram of a classical knot (the diagram in Fig. ~\ref{Gauss1} is an example), every Gauss diagram {\em can} be realized as the diagram of a virtual knot \cite{KauffmanVKT}.  It is often useful to think of virtual knots in terms of Gauss diagrams, since it avoids the pitfalls created by misinterpreting virtual knots as physical objects.

Since every classical or virtual knot is equivalent to the unknot via a sequence of the extended Reidmeister moves together with the {\em forbidden moves}, illustrated in Section \ref{forbidden}, we can consider the minimum number of forbidden moves necessary to unknot a given knot.  We define this invariant, the {\em forbidden number,} in Section \ref{forbidden} and provide an upper bound for the forbidden number of a knot in terms of its crossing number.  We continue in Section \ref{examples} by computing upper bounds on the forbidden number for the Kishino knots with $n$ full twists, as well as for twist knots and torus knots.  In Section \ref{ow} we relate the forbidden number to the odd writhe \cite{KaSelfLink} and odd writhe polynomial \cite{Cheng}, which enables us to establish lower bounds on the forbidden number, and compute several more examples.  Finally, in Section \ref{forbiddensmall} we apply our results to update Sakurai's table \cite{Sakurai} of forbidden numbers for knots with crossing number less than or equal to four.

\section{The forbidden number}\label{forbidden}

As mentioned in the Introduction, there are two additional Reidemeister-like moves for virtual knots, known as the \emph{forbidden moves}, illustrated in Fig. ~\ref{Fmoves}. Move $FO$ moves a strand of the diagram ``over'' a virtual crossing, while move $FU$ moves a strand ``under'' a virtual crossing. 
\begin{figure}[htbp]
\begin{center}
\includegraphics{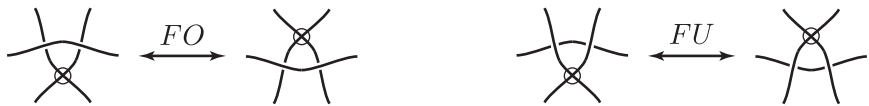}
\end{center}
\caption{Forbidden moves.}
\label{Fmoves}
\end{figure}
Neither of these moves can be obtained as a sequence of extended Reidemeister moves. 

Unlike the other virtual Reidemeister moves, the forbidden moves {\em do} change the Gauss diagram of a virtual knot.  The move $FO$ has the effect of switching the tails of two arrows in a Gauss diagram, while the move $FU$ switches the heads, as shown in Fig. ~\ref{FGauss}.
\begin{figure}[htbp]
\begin{center}
{\includegraphics{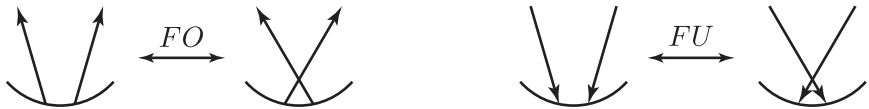}}
\end{center}
\caption{Forbidden moves on Gauss diagrams.}
\label{FGauss}
\end{figure}
Nelson \cite{Nelson} observed that using two forbidden moves (one of each), we can perform the \emph{forbidden detour} move shown in Fig. ~\ref{Fdetour}, which has the effect on Gauss diagrams of switching the head of one arrow with the tail of an adjacent arrow.

\begin{figure}[h]
\begin{center}
\includegraphics{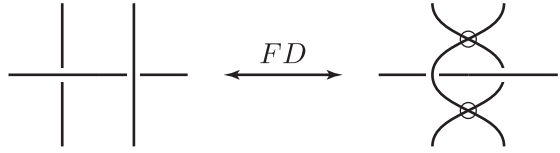}
\end{center}
\caption{Forbidden detour.}
\label{Fdetour}
\end{figure}

If we allow one forbidden move but not the other, we obtain what are known as \emph{welded knots}, developed by Satoh \cite{Satoh} and Kamada \cite{Kamada}. If we allow both forbidden moves, then any virtual knot can be transformed into any other virtual knot (\cite{GPV}, \cite{Kanenobu}, \cite{Nelson}; see also Thm. ~\ref{crossingbound} of this paper), hence the designation of these moves as forbidden.

In particular, every classical or virtual knot can be transformed into the unknot by a sequence of extended Reidemeister moves and forbidden moves. It is, therefore, natural to ask how many forbidden moves are necessary to unknot a given knot.

\begin{defn}\label{forbiddennumber}
Let $K$ be a classical or virtual knot. The \emph{forbidden number} of $K$, $F(K)$, is the minimum number of forbidden moves necessary to transform $K$ into the unknot.
\end{defn}

It is understood that Defn.\ \ref{forbiddennumber} is to be taken over all diagrams for $K$, and that unlimited extended Reidemeister moves are permitted during the transformation. Thus the forbidden number is an invariant of $K$, similar in spirit to the {\em unknotting number} (the minimum number of crossing changes needed to transform $K$ into the unknot) and the {\em virtual unknotting number} (the minimum number of classical crossings that need to be turned into virtual crossings to transform $K$ into the unknot).

It follows immediately that only the unknot has forbidden number zero, since a sequence of extended Reidemeister moves (without forbidden moves) that transforms $K$ into the unknot would place $K$ (by definition) in the same equivalence class as the unknot.

Recall that if $K$ is an oriented virtual knot, the \emph{inverse} of $K$, denoted by $K^{\ast}$, is obtained by reversing the orientation of $K$; the \emph{mirror image} of $K$, denoted by $\bar{K}$, is the result of switching all the classical crossings of $K$.

\begin{thm} $F(K^{\ast}) = F(K) = F\of{\bar{K}}$
\end{thm}
\begin{proof}
A knot and its inverse have identical Gauss diagrams, except the circle is traversed in the opposite direction. This has no bearing on either the extended Reidemeister moves or the forbidden moves. Thus $F(K^*)=F(K)$. To see that $F(K)=F(\bar{K})$, observe that taking a mirror image has the effect of reversing all the arrows in the Gauss diagram. Thus an unknotting sequence of extended Reidemeister moves and forbidden moves can be replaced by an identical sequence of moves, but replacing $FO$ moves with $FU$ and vice versa.
\end{proof}


\section{The forbidden number and the crossing number}

We recall that the {\it crossing number,} $c(K)$, of a knot is the minimum number of crossings of any diagram of $K$. To extend this definition to virtual knots, we simply ignore all virtual crossings. Thus the crossing number of a virtual knot $K$ is the minimum number of classical (i.e., non-virtual) crossings in any diagram of $K$. In this section, we will use the crossing number to derive upper bounds for the forbidden number.  We find these bounds using the forbidden moves and forbidden detours, along with the first Reidemeister move; the following lemma illustrates our method.

\begin{lemma}\label{crossingbound} For any virtual knot $K$, $F(K)\le \frac{c(c-1)}{2}+\left\lfloor\frac{(c-1)^2}{4}\right\rfloor$, where $c$ is the crossing number of $K$.
\end{lemma}
\begin{proof}
From the point of view of Gauss diagrams, we can unknot any virtual knot by moving ends of chords past each other to isolate each chord, and then eliminating chords using the first Reidemeister move.  In a diagram with $c$ crossings, isolating a chord requires moving its head or tail past at most $c-1$ ends of other chords, of which at most $\left\lfloor \frac{c-1}{2} \right\rfloor$ are of the opposite type (i.e., require moving a head past a tail or vice versa).  Moving a head/tail past an end of the same type requires one forbidden move; moving past an end of the opposite type is a forbidden detour which uses two forbidden moves.  So isolating and removing the first arrow requires at most $\left\lfloor \frac{3(c-1)}{2} \right\rfloor$ forbidden moves.  Doing this for each successive arrow (reducing the number of crossings each time) requires a total of at most:
$$\sum_{k=1}^{c-1}\left\lfloor\frac{3k}{2}\right\rfloor=\frac{c(c-1)}{2}+\left\lfloor\frac{(c-1)^2}{4}\right\rfloor$$
forbidden moves, which completes the proof.
\end{proof}

We will improve this bound by examining a particular family of virtual knots we call {\it complete knots}.  A complete knot has a Gauss diagram in which every arrow crosses every other arrow (so the {\it intersection graph} of the Gauss diagram is a complete graph, hence the name). For example, Fig. \ref{F:complete} shows all the possible Gauss diagrams for complete knots with 5 crossings (ignoring the signs of the crossings). Note that the only classical knots which are complete knots are the $(p,2)$-torus knots (for $p$ odd). In these cases, the heads and tails of the arrows alternate around the circle of the Gauss diagram (so the diagram on the left of Fig. \ref{F:complete} is the Gauss diagram for the $(5,2)$-torus knot).

\begin{figure}[htbp]
\begin{center}
\scalebox{1}{\includegraphics{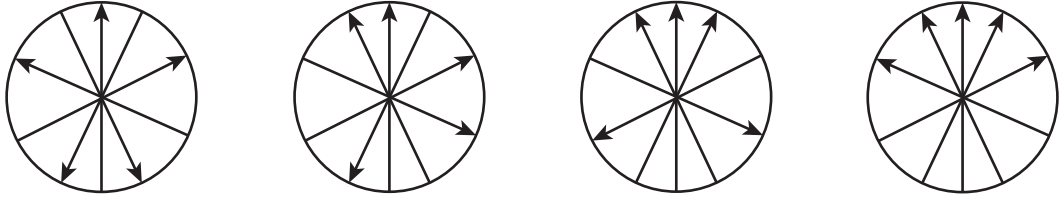}}
\end{center}
\caption{Gauss diagrams for complete knots with 5 crossings.}
\label{F:complete}
\end{figure}

\begin{lemma}\label{completebound}
If $K$ is a complete knot with crossing number $c$, and if $h$ is the length of the longest sequence of consecutive undercrossings in a minimal diagram for $K$, then $F(K)\le \frac{c(c-1)}{2}+\left\lfloor\frac{(c-h)^2}{4}\right\rfloor$.
\end{lemma}
\begin{proof}
As in the proof of Lem. \ref{crossingbound}, we will use forbidden moves and forbidden detours to move ends of arrows past each other, and then remove arrows using the first Reidemeister move. The sequence of consecutive undercrossings corresponds to a sequence of consecutive heads in the Gauss diagram for $K$.  Observe that $h \neq c-1$, since if there is a sequence of $c-1$ consecutive heads, the last arrow will extend this to a sequence of length $c$.  So suppose $h \leq c-2$.  Let $A$ and $B$ be the arrows whose tails bound the sequence of $h$ heads. One of $A$ or $B$ will have at least $h + \left\lceil \frac{c-h-2}{2} \right\rceil$ heads on one side, and at most $1+\left\lfloor \frac{c-h-2}{2} \right\rfloor$ tails (the additional 1 comes from the other of chords $A$ and $B$).  So the total number of forbidden moves required to remove that arrow will be:
$$h + \left\lceil \frac{c-h-2}{2} \right\rceil + 2 + 2\left\lfloor \frac{c-h-2}{2} \right\rfloor = h+2 + \left\lfloor \frac{3}{2}(c-h-2) \right\rfloor = (h-1) + \left\lfloor \frac{3}{2}(c-h) \right\rfloor.$$

Since removing either $A$ or $B$ does not decrease the length of the sequence of heads, the resulting knot will still be a complete knot with a maximal sequence of length at least $h$.  So we can continue this process until there are only $h+1$ arrows left, at which point all the heads will be consecutive. After that, no more forbidden detours are needed.  So the total number of forbidden moves required is:
$$\sum_{k=2}^{c-h}{\left(h-1 + \left\lfloor \frac{3k}{2} \right\rfloor\right)} + \sum_{k=1}^h{k} = \frac{c(c-1)}{2}+\left\lfloor\frac{(c-h)^2}{4}\right\rfloor.$$
\end{proof}

Table \ref{completebounds} lists the upper bound given by Lem. \ref{completebound} for small values of $c$ and $h$. Note that when $h = 1$, the bounds are those given by Lem. \ref{crossingbound}.

\begin{table}[htbp]
\begin{center}
\begin{tabular}{cccccccccccc}
& & \multicolumn{10}{c}{crossing number $c$} \\ 
& & 1 & 2 & 3 & 4 & 5 & 6 & 7 & 8 & 9 & 10 \\ \cline{3-12}
\multicolumn{1}{c}{\multirow{10}{*}{$h$}} & \multicolumn{1}{c|}{1} & 0 & 1 & 4 & 8 & 14 & 21 & 30 & 40 & 52 & 65 \\
\multicolumn{1}{c}{} & \multicolumn{1}{c|}{2} & & 1 & 3 & 7 & 12 & 19 & 27 & 37 & 48 & 61 \\
\multicolumn{1}{c}{} & \multicolumn{1}{c|}{3} & & & 3 & 6 & 11 & 17 & 25 & 34 & 45 & 57 \\
\multicolumn{1}{c}{} & \multicolumn{1}{c|}{4} & & & & 6 & 10 & 16 & 23 & 32 & 42 & 54 \\
\multicolumn{1}{c}{} & \multicolumn{1}{c|}{5} & & & & & 10 & 15 & 22 & 30 & 40 & 51 \\
\multicolumn{1}{c}{} & \multicolumn{1}{c|}{6} & & & & & & 15 & 21 & 29 & 38 & 39 \\
\multicolumn{1}{c}{} & \multicolumn{1}{c|}{7} & & & & & & & 21 & 28 & 37 & 47 \\
\multicolumn{1}{c}{} & \multicolumn{1}{c|}{8} & & & & & & & & 28 & 36 & 46 \\
\multicolumn{1}{c}{} & \multicolumn{1}{c|}{9} & & & & & & & & & 36 & 45 \\
\multicolumn{1}{c}{} & \multicolumn{1}{c|}{10} & & & & & & & & & & 45 \\
\end{tabular}
\end{center}
\caption{Upper bounds on $F(K)$ for complete knots with $c$ crossings, and at least $h$ consecutive undercrossings.}
\label{completebounds}
\end{table}

\begin{thm}\label{upperbound}
For any virtual knot $K$ with crossing number $c$, $F(K) \leq \left\lfloor \frac{3c^2-6c+7}{4} \right\rfloor$.
\end{thm}
\begin{proof}
We first consider the case when $K$ is a complete knot. If $c$ is even, then there must be two consecutive undercrossings in the diagram for $K$. Otherwise, the diagram would alternate between overcrossings and undercrossings. If we begin at a crossing $A$, this means that we need to go over an even number of crossings before returning to $A$ (otherwise we would try to go over or under $A$ twice).  But since the knot is complete and $c$ is even, the diagram passes through an odd number of crossings before returning to $A$, which is impossible.

So, if $K$ is a complete knot and $c$ is even, then by Lem. \ref{completebound} we have 
$$F(K) \leq  \frac{c(c-1)}{2}+\left\lfloor\frac{(c-2)^2}{4}\right\rfloor = \frac{c(c-1)}{2}+\frac{(c-2)^2}{4} = \frac{3c^2-6c+4}{4} = \left\lfloor \frac{3c^2-6c+7}{4} \right\rfloor.$$ 
On the other hand, if $c$ is odd, then removing the first arrow from the Gauss diagram requires at most $\frac{3}{2}(c-1)$ forbidden moves, as in Lem. \ref{crossingbound} (since $c-1$ is even, the floor function is unnecessary). The result is a complete knot with $c-1$ crossings, which must then have two consecutive undercrossings as above.  So in this case we have 
$$F(K) \leq \frac{3}{2}(c-1)+\frac{(c-1)(c-2)}{2}+\left\lfloor\frac{(c-3)^2}{4}\right\rfloor = \frac{3}{2}(c-1)+\frac{(c-1)(c-2)}{2}+\frac{(c-3)^2}{4} = \frac{3c^2-6c+7}{4}.$$

If $K$ is {\it not} complete, then there are two arrows in the Gauss diagram which do not intersect. Then at least one of them has at most $c-2$ heads and tails on one side, so can be removed with at most $\left\lfloor \frac{3}{2}(c-2) \right\rfloor$ forbidden moves.  On the other hand, from the proof of Lem. \ref{completebound}, removing an arrow from a complete knot with two consecutive undercrossings requires at most $1 + \left\lfloor \frac{3}{2}(c-2) \right\rfloor$ forbidden moves. So the bound computed for the complete knots also bounds the forbidden number for any non-complete knots.
\end{proof}

Below is a table listing the upper bound provided by Thm. \ref{upperbound} for crossing numbers up to 12.

\mathstrut\begin{center}
\begin{tabular}{| l | c  c  c  c  c  c  c  c  c  c  c |}
\hline
Crossing number & 2 & 3 & 4 & 5 & 6 & 7 & 8 & 9 & 10 & 11 & 12 \\ \hline
Upper bound for $F(K)$ & 1 & 4 & 7 & 13 & 19 & 28 & 37 & 49 & 61 & 76 & 91 \\
\hline
\end{tabular}
\end{center}\mathstrut

\begin{example}\label{Ex:torus} The bounds in Lem. \ref{completebound} and Thm. \ref{upperbound} are far from sharp. In particular, as arrows are removed from a complete knot, the sequence of consecutive undercrossings is likely to increase in length. To illustrate this, we will consider the $(p,2)$-torus knots $T_{p,2}$ (where $p$ is odd), which are the only classical complete knots. The Gauss diagrams for these knots consist of $p$ arrows, all with positive sign, with heads and tails alternating around the diagram. If we follow the procedure in the proof of Lem. \ref{completebound}, then each time we remove an arrow we increase the length of the longest sequence of undercrossings by 1, as shown for the $(5,2)$-torus knot in Fig. \ref{F:5torus}.

\begin{figure}[htbp]
\begin{center}
\scalebox{1}{\includegraphics{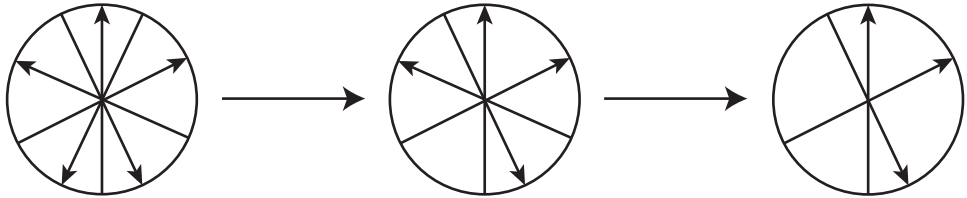}}
\end{center}
\caption{Removing arrows from the $(5,2)$-torus knot.}
\label{F:5torus}
\end{figure}

When there are $\frac{p+1}{2}$ arrows left, all the remaining heads are consecutive.  So:
$$F(T_{p,2}) \leq \sum_{k=1}^{\frac{p-1}{2}}{\left(k-1 + \frac{3}{2}(p-2k+1)\right)} + \sum_{k=1}^{\frac{p-1}{2}}{k} = \frac{5p^2 - 4n-1}{8}$$
For example, $F(T_{11,2}) \leq 70$, compared to the bound of 76 from Thm. \ref{upperbound}.
\end{example}



\section{Examples}\label{examples}

In this section we will examine several common families of knots: the Kishino knot with twists added, twist knots and torus knots.  For the Kishino knots, we prove the forbidden number for all such knots is 1; for the other families, we will provide upper bounds for the forbidden number.  In these examples, the upper bounds are found through direct analysis of the Gauss diagrams, though in the first example, the corresponding moves in the knot diagram are easy to see.


\begin{figure}[htbp]
\begin{center}
\includegraphics{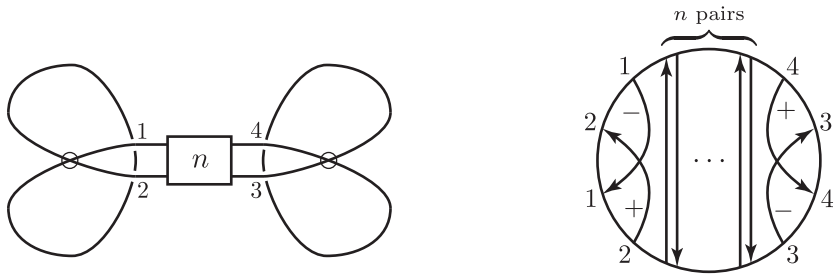}
\end{center}
\caption{Kishino Knot with $n$ full twists.}
\label{Kishino-n}
\end{figure}

\subsection*{Kishino knot with $n$ full twists.}  Fig. ~\ref{Kishino-n} shows the Kishino knot with $n$ full twists, denoted $K_n$ (the original Kishino knot is the case when $n = 0$).  For all $n\ge0$, we have $F(K_n) = 1$. The knots $K_n$ are distinguished from each other (and from the unknot) by the arrow polynomial \cite{Henrichpseudo}. Thus $F(K_n)\ge1$. 

To show that $F(K_n)=1$ we perform the following moves:  We begin by using forbidden move $FU$ to slide the head of arrow 1 past the head of arrow 2. Next we remove crossings 1 and 2 with $R1$ moves. Then the $2n$ crossings in the middle can be eliminated with $R1$ moves. Finally we use an $R2$ move to eliminate crossings 3 and 4. Note that the collection $\{K_n\}$ provides an infinite family of distinct knots with forbidden number one.

\subsection*{\bf Twist Knots.} The twist knot $T_n$ with $n$ half-twists is shown in Fig. ~\ref{twistknot}, along with the Gauss diagram for the two cases: when $n$ is even and when $n$ is odd.  All twist knots have unknotting number one and virtual unknotting number two \cite{fm}, but their forbidden number most likely depends on $n$.  If $n$ is odd then $F(T_n) \leq 3n + 1$.  This is achieved by using three forbidden moves to move one end of each of the horizontal chords past the vertical chords, and then removing the horizontal chord by an $R1$ move; a final forbidden move is then used to uncross the two vertical chords.   

If $n$ is even, then $F(T_n) \leq \frac{5}{2}n - 1$.  In this case, we move the tail of the right-hand vertical chord past all the horizontal chords except the top one (requiring $(n-1) + \frac{n}{2} = \frac{3}{2}n - 1$ forbidden moves), move the head of the right-hand vertical chord past the tail of the other vertical chord (requiring two forbidden moves), and then remove the (originally) right-hand vertical chord and the top horizontal chord by an $R2$ move.  Now we can remove all but one of the remaining horizontal chords, one at a time, by combining a forbidden move and an $R1$ move; the last two chords are removed by another $R2$ move.  This is a total of $\frac{3}{2}n - 1 + 2 + (n-2) = \frac{5}{2}n-1$ forbidden moves.  In particular, the figure-eight knot is $T_2$, so this gives us an upper bound of four for the forbidden number of the figure-eight knot.

\begin{figure}[htbp]
\begin{center}
\scalebox{1}{\includegraphics{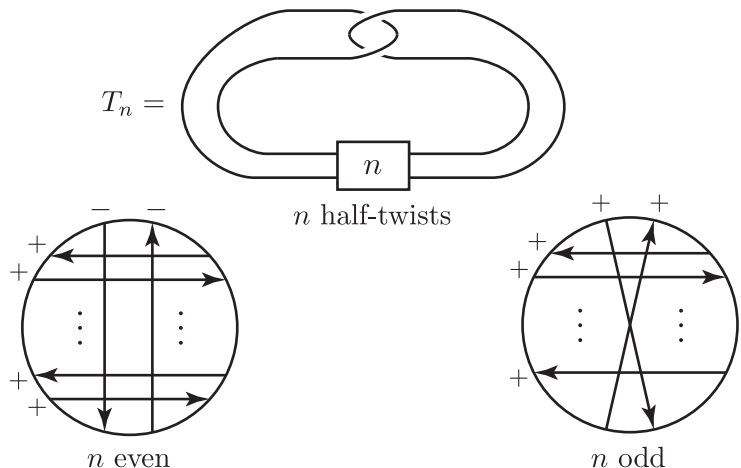}}
\end{center}
\caption{Twist knot $T_n$.}
\label{twistknot}
\end{figure}

\subsection*{Torus Knots} \label{torusknots}

In Ex. ~\ref{Ex:torus}, we found an upper bound for the forbidden number of the $(p,2)$-torus knot ($p$ odd). That bound was created from a diagram with a minimal number of crossings. But a better bound can be obtained using a diagram with \emph{more} crossings. The $(p,2)$-torus knot is a $2$-bridge knot, meaning it has a diagram with only two overpasses (segments without undercrossings) and two underpasses (segments without overcrossings) appearing alternately along the diagram. Figure \ref{torus52} is a $2$-bridge presentation for the $(5,2)$-torus knot with the corresponding Gauss diagram.
\begin{figure}[htbp]
\begin{center}
$$\includegraphics{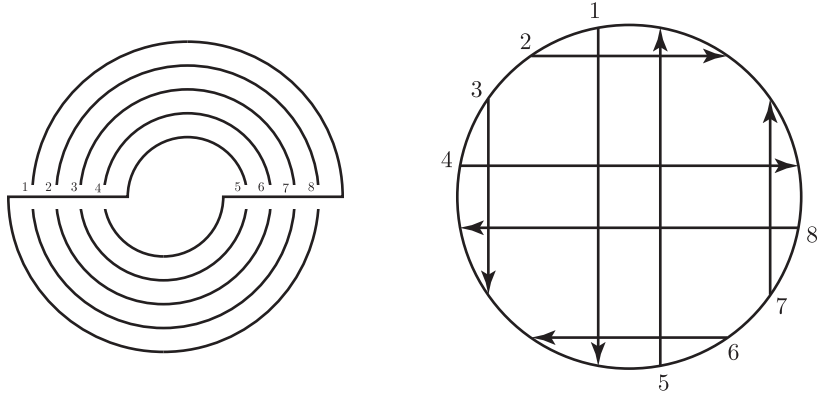}$$
\end{center}
\caption{$(5,2)$-torus knot and Gauss diagram.}
\label{torus52}
\end{figure}

Recall that forbidden moves move heads past adjacent heads or tails past adjacent tails. Thus crossing 2 can be eliminated from the diagram with two forbidden moves (one $FO$ and one $FU$) followed by an $R1$ move, and similarly for crossing 6. Crossings 4 and 8 can then be eliminated with four forbidden moves each (two $FO$ and two $FU$) and $R1$ moves. The remaining crossings (1, 3, 5, 7) are eliminated with $R1$ moves. The total number of forbidden moves is $2(2+4)=12$. Thus $F(T_{5,2})\le 12$.  In contrast, the formula in Ex. \ref{Ex:torus} gives an upper bound of 13.

The situation is similar for $T_{p,2}$. The $2$-bridge presentation will have $2(p-1)$ crossings, with one overpass numbered from $1$ to $p-1$, and the other from $p$ to $2(p-1)$. The Gauss diagram will have the even chords horizontal and the odd chords vertical. The even chords from $2$ to $p-1$ can be eliminated sequentially with $2$ to $p-1$ forbidden moves respectively, followed by $R1$ moves. Symmetrically, the remaining chords can be eliminated with the same number of forbidden moves. The total number of forbidden moves used is $2\big(2+4+\cdots + (p-1)\big) = \frac{p^2-1}{2}$. Thus we have proven the following: 

\begin{thm} \label{torusbound}
$F(T_{p,2})\le\frac{p^2-1}{2}$.
\end{thm}



Note that the better bound for the forbidden number is obtained from a diagram that does not have a minimal number of crossings: $2(p-1)$ crossings in the 2-bridge presentation as compared to $p$ crossings in the minimal presentation.


\section{Forbidden number and Odd writhe}\label{ow}

We recall that a crossing of a virtual knot is called {\em odd} if its crossing labels in the Gauss diagram have an odd number of crossing labels between them. The {\em odd writhe}, then, is the sum of the signs of the odd crossings.  Kauffman showed that the odd writhe is invariant under the extended Reidemeister moves, and hence is an invariant of virtual knots \cite{KaSelfLink}.  We note that in classical knots the odd writhe is always zero, since all crossings of classical knots are even.  For the virtual trefoil shown in Fig. ~\ref{Gauss1}, the odd writhe is two, proving that this knot is not equivalent to any classical knot.

While the odd writhe is invariant under the extended Reidemeister moves, it is not invariant under the forbidden moves.

\begin{thm} \label{oddwrithe}
A forbidden move changes the odd writhe by 0, 2 or -2.
\end{thm}

\begin{proof}
A forbidden move either crosses or uncrosses two arrows of the Gauss diagram, causing their parities to change.  If they have the same parity and the same sign or opposite parities and opposite signs, the odd writhe will change by $\pm 2$; if they have the same parity and opposite signs or opposite parities and the same sign, the odd writhe will be unchanged.
\end{proof}

This result immediately gives us the following lower bound on the forbidden number.

\begin{thm} \label{lowerbound}
For any virtual knot $K$, $F(K) \geq \left\vert \frac{1}{2}\oddw(K)\right\vert$, where $\oddw(K)$ is the odd writhe of $K$.
\end{thm}

We remark that this tells us nothing about the forbidden number of classical knots since they have odd writhe of zero. 

As an application, we can compute the forbidden number for a ring of virtual trefoil knots.  In particular, the forbidden number of the virtual trefoil is one.

\begin{thm} The forbidden number of a ring of $n$ virtual trefoil knots is $n$. 
\end{thm}

\begin{proof}
The Gauss diagram for a ring of $n$ virtual trefoil knots is shown in Fig. ~\ref{TrefoilRing}.
\begin{figure}[htbp]
\begin{center}
\scalebox{.5}{\includegraphics{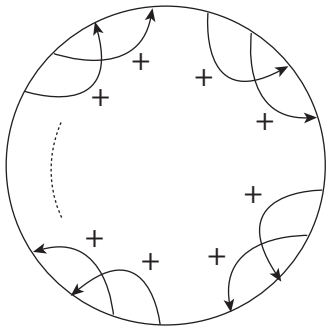}}
\end{center}
\caption{A ring of $n$ virtual trefoil knots.}
\label{TrefoilRing}
\end{figure}
Each pair of crossed arrows can be uncrossed by a single forbidden move, and the crossings can then all be removed using the first Reidemeister move, so the forbidden number is at most $n$.

On the other hand, each of the $2n$ crossings is odd, and they all have the same sign, so the odd writhe is $2n$.  So, by Thm. ~\ref{lowerbound}, the forbidden number is at least $n$.  Therefore, the forbidden number is exactly $n$.
\end{proof}


\begin{example} {\bf Virtual twist knots.}  The {\em virtual twist knots $VT_n$} are the twist knots $T_n$ with the top crossing made virtual, as shown in Fig. ~\ref{virtualtwistknot}.  If $n$ is odd, then the odd writhe is $n + 1$ so $F(VT_n) \geq \frac{n+1}{2}$; on the other hand, we can remove each horizontal chord in turn with a single forbidden move and an R1 move, so $F(VT_n) \leq n$.  

If $n$ is even then the odd writhe is $n$ and $F(VT_n) \geq \frac{n}{2}$.  We remove the horizontal chords as when $n$ is odd, except that the last horizontal chord and the vertical chord can be removed together by an $R2$ move.  So when $n$ is even $F(VT_n) \leq n-1$.
\end{example}

\begin{figure}[htbp]
\begin{center}
\scalebox{1}{\includegraphics{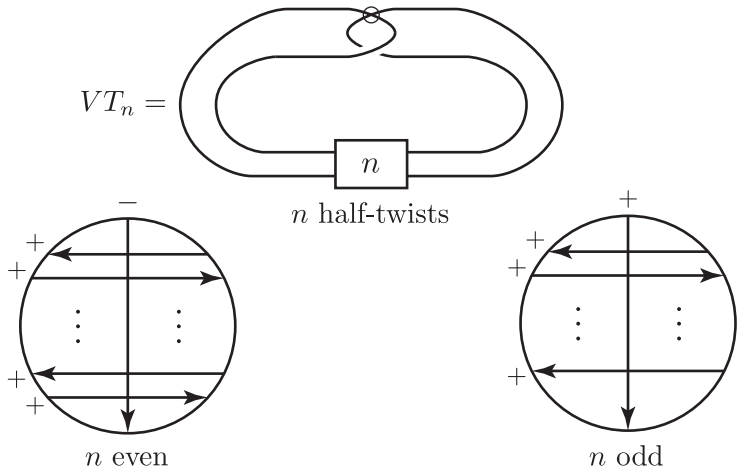}}
\end{center}
\caption{Virtual twist knot $VT_n$.}
\label{virtualtwistknot}
\end{figure}

\subsection*{Forbidden Number and the Odd Writhe Polynomial}

Cheng \cite{Cheng} extended the idea of the odd writhe, creating the \emph{odd writhe polynomial}, denoted $\OWP \in \Z[t, t^{-1}]$. This polynomial has been completely characterized and generalizes the odd writhe invariant; the sum of the coefficients of \OWP\ is precisely the odd writhe.  We will briefly review the definition of the odd writhe polynomial, and then describe how it changes under forbidden moves.

Consider a virtual knot diagram $K$ with real crossings $c_1, \dots, c_n$.  The Gauss diagram for $K$ contains $n$ arrows corresponding to these crossings, which divide the circle into $2n$ arcs.  We assign each arc an integer as follows:  beginning at a point in the arc, we follow the orientation of the circle (traditionally counter-clockwise) around the Gauss diagram.  For each arrow in the diagram, we will first encounter either its {\em head}  (corresponding to an undercrossing) or its {\em tail} (corresponding to an overcrossing).  The integer assigned to the arc is the sum of the signs of the arrows whose {\em heads} are encountered first.

Next we use the integers on the arcs to assign an integer to each arrow $c_i$.  The endpoints of $c_i$ are incident to four arcs of the Gauss circle; two arcs at the head and two arcs at the tail.  We denote the labels of these arcs by $h_1$ and $h_2$ at the head and $t_1$ and $t_2$ at the tail.  Then the integer assigned to $c_i$ is $N(c_i) = \max\{h_1, h_2\} - \min\{t_1,t_2\}$.  Finally, we define the {\em odd writhe polynomial} of $K$ to be:
$$\OWP = \sum_{c_i\ odd}{sign(c_i)t^{N(c_i)}}$$
\noindent where the sum is over all odd crossings in the knot diagram.  Note that classical knots, which have no odd crossings, have $\OWP = 0$.

\begin{example} We compute the odd writhe polynomial for the virtual knot $K$ in Fig. ~\ref{knot447} (knot 4.47 from \cite{Green}).

\begin{figure}[htbp]
\begin{center}
{\includegraphics{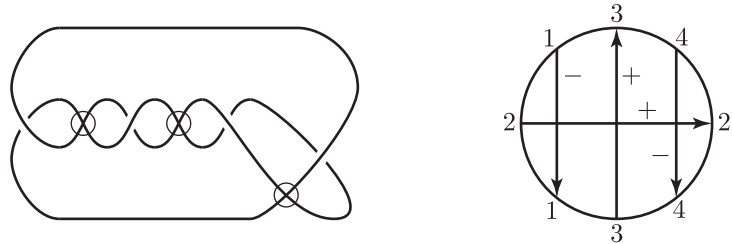}}
\end{center}
\caption{Virtual knot 4.47.}
\label{knot447}
\end{figure}
The labels for the arcs of the Gauss diagram are shown in Fig. ~\ref{knot447b}, along with the computation of $N(c_i)$.  Since all of the crossings are odd, the odd writhe polynomial is:
$$\OWP = -t^2 + t^4 + t^0 - t^2 = t^4 - 2t^2 + 1$$

\begin{figure}[htbp]
{\includegraphics{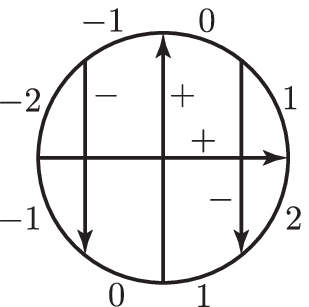}} \parbox[b]{1in}{\begin{align*} 
N(c_1) &= 0 - (-2) = 2 \\ 
N(c_2) &= 2 - (-2) = 4 \\ 
N(c_3) &= 0 - 0 = 0 \\
N(c_4) &= 2 - 0 = 2
\end{align*}\vspace{.25 ex}} 
\caption{Labeling the arcs and arrows of the Gauss diagram.}
\label{knot447b}
\end{figure}

\end{example}

We can use the odd writhe polynomial to establish lower bounds on the forbidden number of many virtual knots, by seeing how it is changed by a forbidden move.

\begin{thm} \label{T:OWP}
Performing a forbidden move on a knot $K$ changes \OWP\ by $\pm t^m \pm t^n$ for some $m$ and $n$. (It is possible that $m = n$.)
\end{thm}

\begin{proof}
Figure \ref{F:owpfu} shows the result of performing an $FU$ move on a Gauss diagram, and how it affects the labels on the arcs of the diagram.  The signs of the crossings, $\pm 1$, are represented by $\e$ and $\de$.  The only arc that changes its label is the arc between the heads of the two arrows involved.  It is also important to note that the parities of the two crossings (odd or even) will change.

\begin{figure}[htbp]
\begin{center}
\scalebox{.75}{\includegraphics{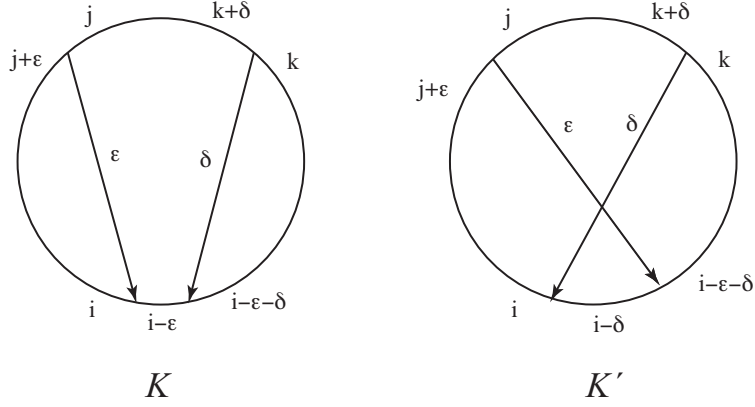}}
\end{center}
\caption{Performing a forbidden $FU$ move.}
\label{F:owpfu}
\end{figure}

Let $c_\e$ and $c_\de$ denote the crossings in knot $K$ with signs $\e$ and $\de$, respectively; let $c'_\e$ and $c'_\de$ denote the corresponding crossings in $K'$.  Then in $K$, $N(c_\e) = i - j + (1-\e)$ and $N(c_\de) = i-k + (1-\e-\de)$.  In $K'$, $N(c'_\e) = i - j + (1-\e-\de)$ and $N(c'_\de) = i-k + (1-\de)$.  The labels (and parities) of all other crossings in the knots are the same.  Now we have four cases, which are summarized in the table below:

\begin{center}
\begin{tabular}{l|c|c}
$W_K - W_{K'}$ & $c_\e$ is odd & $c_\e$ is even \\ \hline
$c_\de$ is odd & \parbox[c][1.5\height]{1.7in}{$\e t^{N(c_\e)} + \de t^{N(c_\de)} \\= \e t^{i-j+1-\e} + \de t^{i-k+1-\e-\de}$} & \parbox{1.9in}{$-\e t^{N(c'_\e)} + \de t^{N(c_\de)} \\= -\e t^{i-j+1-\e-\de} + \de t^{i-k+1-\e-\de}$} \\ \hline
$c_\de$ is even & \parbox[c][1.5\height]{1.7in}{$\e t^{N(c_\e)} - \de t^{N(c'_\de)} \\= \e t^{i-j+1-\e} - \de t^{i-k+1-\de}$} & \parbox{1.8in}{$-\e t^{N(c'_\e)} - \de t^{N(c'_\de)} \\= -\e t^{i-j+1-\e-\de} - \de t^{i-k+1-\de}$}
\end{tabular}
\end{center}

We now perform a similar analysis for an $FO$ move.  Figure \ref{F:owpfo} shows the result of performing an $FO$ move.

\begin{figure}[htbp]
\begin{center}
\scalebox{.75}{\includegraphics{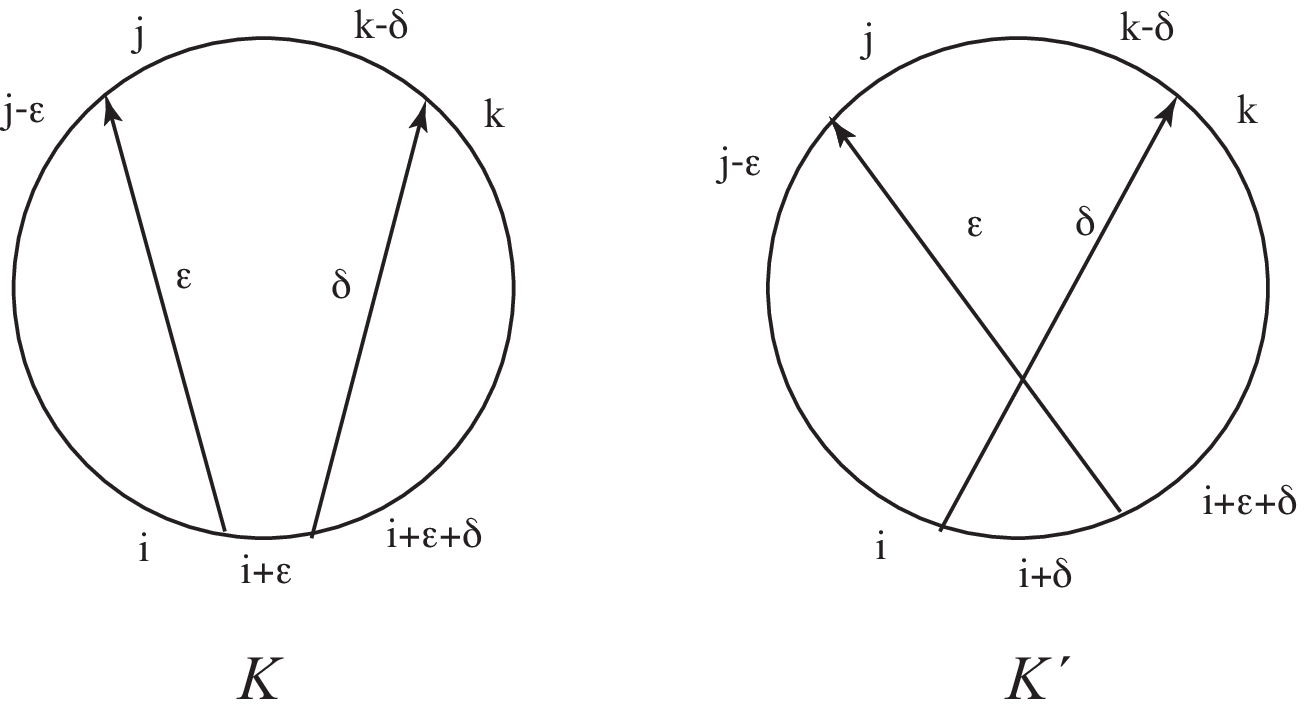}}
\end{center}
\caption{Performing a forbidden $FO$ move.}
\label{F:owpfo}
\end{figure}
As before, let $c_\e$ and $c_\de$ denote the crossings in knot $K$ with signs $\e$ and $\de$, respectively; let $c'_\e$ and $c'_\de$ denote the corresponding crossings in $K'$.  Then in $K$, $N(c_\e) = j-i + (1-\e)$ and $N(c_\de) = k-i + (1-\e-\de)$.  In $K'$, $N(c'_\e) = j-i + (1-\e-\de)$ and $N(c'_\de) = k-i + (1-\de)$.  The labels (and parities) of all other crossings in the knots are the same.  Now we have four cases, which are summarized in the table below:

\begin{center}
\begin{tabular}{l|c|c}
$W_K - W_{K'}$ & $c_\e$ is odd & $c_\e$ is even \\ \hline
$c_\de$ is odd & \parbox[c][1.5\height]{1.7in}{$\e t^{N(c_\e)} + \de t^{N(c_\de)} \\= \e t^{j-i+1-\e} + \de t^{k-i+1-\e-\de}$} & \parbox{1.9in}{$-\e t^{N(c'_\e)} + \de t^{N(c_\de)} \\= -\e t^{j-i+1-\e-\de} + \de t^{k-i+1-\e-\de}$} \\ \hline
$c_\de$ is even & \parbox[c][1.5\height]{1.7in}{$\e t^{N(c_\e)} - \de t^{N(c'_\de)} \\= \e t^{j-i+1-\e} - \de t^{k-i+1-\de}$} & \parbox{1.9in}{$-\e t^{N(c'_\e)} - \de t^{N(c'_\de)} \\= -\e t^{j-i+1-\e-\de} - \de t^{k-i+1-\de}$}
\end{tabular}
\end{center}
In every case, we have $\OWP - W_{K'}(t) = \pm t^m \pm t^n$ for some integers $m$ and $n$.

\end{proof}

We remark that $\OWP - W_{K'}(t) = 0$ only in the following circumstances: \begin{enumerate}
	\item $\e = \de$, the crossings have opposite parity, and $j = k$
	\item $\e = -\de$, the crossings have the same parity, and $j = k+\e$ (for an $FU$ move) or $j = k-\e$ (for an $FO$ move)
\end{enumerate}

\begin{cor} \label{C:owp}
If $\OWP = \sum{b_i t^i}$, then $F(K) \geq \frac{1}{2}\sum{\vert b_i\vert}$.
\end{cor}
\begin{proof}
By Thm. ~\ref{T:OWP}, any forbidden move changes \OWP\ by $\pm t^m \pm t^n$, and hence changes $\sum{\vert b_i\vert}$ by at most 2.  Thus, $F(K) \geq \frac{1}{2}\sum{\vert b_i\vert}$.
\end{proof}

\begin{example}\label{E:knot447} Consider the virtual knot $K$ in Fig. ~\ref{knot447} (knot 4.47 from \cite{Green}). The odd writhe is zero, but $\OWP = t^4-2t^2+1$, so the forbidden number is at least two.
We now show that the forbidden number is exactly two. Referring to the Gauss diagram in Fig.~\ref{knot447} we use a forbidden move $FU$ to move the head of 4 past the head of 2. Then we use $R1$ to remove crossing 4. Next we move the head of 3 past the head of 2 using a second $FU$ move. Then, we remove crossing 3 with $R1$. Finally, we remove crossings 1 and 2 with $R2$.

\end{example}





\section{Forbidden Numbers of Knots with Small Crossing Number}\label{forbiddensmall}

The forbidden numbers of many knots with crossing number $\le4$ were computed by Sakurai \cite{Sakurai} by analyzing the effect of forbidden moves on a polynomial defined by Henrich \cite{Henrich}. In this section,  we update and expand the table given in \cite{Sakurai}, incorporating the findings from the present paper.  The results are shown in Table \ref{Ta:forbidden}.  The table lists: \begin{itemize}
\item[$H:$] the lower bound derived by Sakurai from Henrich's polynomial
\item[$OW:$] the lower bound given by Corollary \ref{C:owp} from the odd writhe polynomial
\item[$S:$] the value of $F(K)$ given by Sakurai (when provided)
\item[$F(K):$] the forbidden number of the knot. When the forbidden number has not yet been determined, the best known bounds are listed in the format {\it lower bound--upper bound}.
\end{itemize}

There is considerable overlap between the bounds obtained using the two polynomials (Henrich vs.\ odd writhe). In addition to analyzing all virtual knots with 4 or fewer crossings, we compared the lower bounds from Henrich's polynomial and the odd writhe polynomial for all virtual knots with 5 or 6 crossings (using a computer program and the list of Gauss codes from \cite{Green}).  The results are shown in Table \ref{Ta:census}. 

\begin{table}[htbp]
\begin{center}
\begin{tabular}{|c|c|c|c|c|}
\hline
Crossing number & $H > OW$ & $OW > H$ & $H = OW$ & Total \\ \hline
2 & 0 & 0 & 1 & 1 \\
3 & 3 & 0 & 4 & 7 \\
4 & 29 & 2 & 77 & 108 \\
5 & 3312 & 88 & 5830 & 9230 \\
6 & 143,686 & 10,218 & 200,774 & 354,678 \\ \hline
\end{tabular}
\end{center}
\caption{Comparison of lower bounds from Henrich and odd writhe polynomials} \label{Ta:census}
\end{table}

Among virtual knots with four crossings, the odd writhe polynomial gives a better lower bound for knots 4.26 and 4.47, and we were able to determine the forbidden number. In several other cases, we were able to determine the forbidden number by finding an unknotting sequence that realized the known lower bound, as in Ex. ~\ref{E:knot447}. An asterisk ($*$) indicates the forbidden number is determined by the methods of the present paper, i.e., these knots were not listed in Sakurai's table. Gauss diagrams for these knots are shown in Figure~\ref{newknots}, and their unknotting sequences are given in Table \ref{Ta:sequences}. When the forbidden number has not yet been determined, the best known bounds are listed in the format {\it lower bound--upper bound}. Upper bounds are computed in a manner similar to Ex. ~\ref{E:knot447}. The numbering corresponds to \cite{Green}.

\begin{figure}[htbp]
\begin{center}
\scalebox{1}{\includegraphics{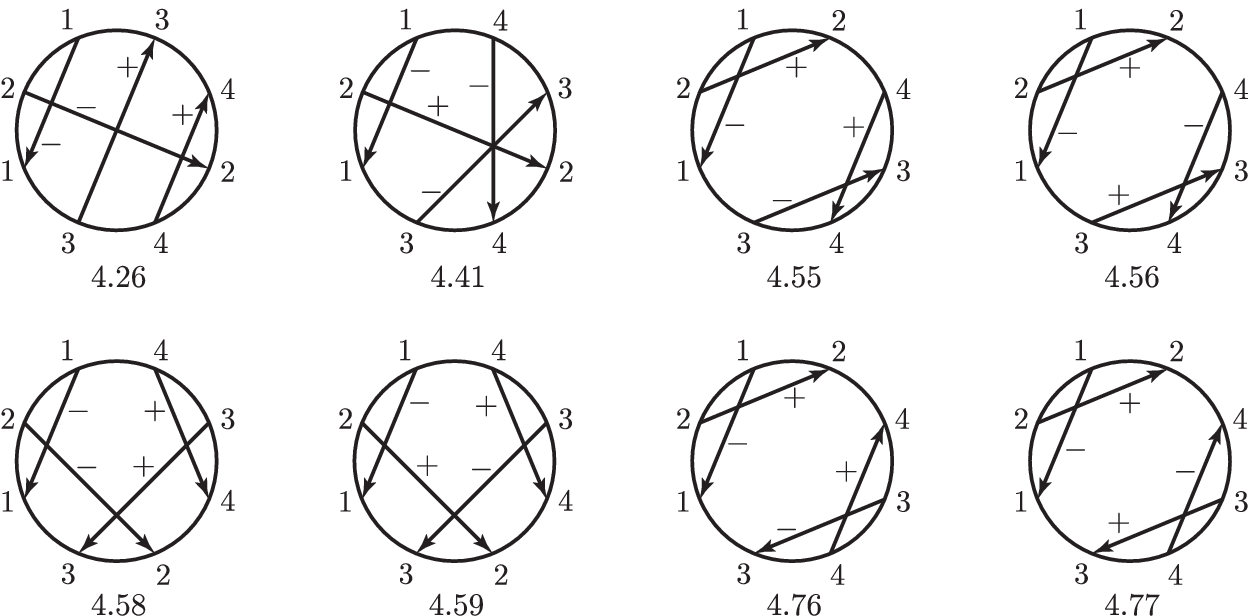}}
\end{center}
\caption{Gauss diagrams.}
\label{newknots}
\end{figure}

To interpret the notation in the unknotting sequences in Table \ref{Ta:sequences}, consider the first example: To unknot virtual knot 4.26, switch the tails of arrows 1 and 2 with a forbidden over move. Then switch the heads of 2 and 4 with a forbidden under move. Remove 1 with a Reidemeister move ($R1$). Then remove 4 ($R1$). Finally, remove 2 and 3 ($R2$).

\begin{table}[htbp]
\begin{center}
\begin{tabular}{|c|l|}
\hline
Knot & Unknotting sequence \\ \hline
4.26 & $FO(1,2)$, $FU(2,4)$, $R1(1)$, $R1(4)$, $R2(2,3)$ \\
4.41 & $FO(1,2)$, $R1(1)$, $R2(2,4)$, $R1(3)$ \\
4.55 & $FO(1,2)$, $R1(1)$, $R1(2)$, $R2(3,4)$ \\
4.56 & Same as 4.55 \\
4.58 & Same as 4.41 \\ 
4.59 & $FU(2,3)$, $R2(1,2)$, $R2(3,4)$ \\
4.76 & Same as 4.55 \\
4.77 & Same as 4.55 \\ \hline
\end{tabular}
\end{center}
\caption{Unknotting sequences for some 4-crossing knots.} \label{Ta:sequences}
\end{table}



 





\begin{table}[htbp]
\begin{center}
\begin{tabular}{| l c c c c | l c c c c | l c c c c |}
\hline
\ $K$ & $H$ & $OW$ & $S$ & $F(K)$ & \ $K$ & $H$ & $OW$ & $S$ & $F(K)$ & \ $K$ & $H$ & $OW$ & $S$ & $F(K)$\\ \hline
0.1  & 0 & 0 & 0 & 0 & 4.31 & 0 & 0 & & \B{1}{2} & 4.70 & 1 & 1 & & \B{1}{2}  \\
2.1  & 1 & 1 & 1 & 1 & 4.32 & 1 & 1 & 1 & 1 & 4.71 & 0 & 0 & & \B{1}{2} \\
3.1  & 1 & 1 & 1 & 1 & 4.33 & 1 & 1 & 1 & 1 & 4.72 & 0 & 0 & & \B{1}{2} \\
3.2  & 1 & 1 & 1 & 1 & 4.34 & 1 & 1 & 1 & 1 & 4.73 & 2 & 2 & 2 & 2  \\
3.3  & 2 & 1 & 2 & 2 & 4.35 & 1 & 1 & 1 & 1 & 4.74 & 1 & 1 & 1 & 1  \\
3.4  & 1 & 1 & 1 & 1 & 4.36 & 2 & 0 & 2 & 2 & 4.75 & 0 & 0 & & \B{1}{2}  \\
3.5  & 2 & 0 & & \B{2}{3} & 4.37 & 3 & 1 & 3 & 3 & 4.76* & 0 & 0 & & 1  \\
3.6  & 0 & 0 & & \B{1}{4} & 4.38 & 1 & 1 & 1 & 1 & 4.77* & 0 & 0 & & 1  \\
3.7  & 2 & 0 & & \B{2}{3} & 4.39 & 1 & 1 & 1 & 1 & 4.78 & 3 & 1 & & \B{3}{4}  \\
4.1  & 2 & 2 & 2 & 2 & 4.40 & 1 & 1 & 1 & 1 & 4.79 & 1 & 1 & 1 & 1  \\
4.2  & 0 & 0 & & \B{1}{2} & 4.41* & 0 & 0 &  & 1 & 4.80 & 3 & 2 & 3 & 3  \\
4.3  & 2 & 2 & 2 & 2 & 4.42 & 1 & 1 & 1 & 1 & 4.81 & 2 & 2 & 2 & 2  \\
4.4  & 1 & 1 & 1 & 1 & 4.43 & 2 & 2 & 2 & 2 & 4.82 & 3 & 1 & 3 & 3 \\
4.5  & 1 & 1 & 1 & 1 & 4.44 & 1 & 1 & & \B{1}{2} & 4.83 & 2 & 1 & 2 & 2 \\
4.6  & 0 & 0 & & \B{1}{2} & 4.45 & 2 & 2 & 2 & 2 & 4.84 & 1 & 1 & & \B{1}{2} \\
4.7  & 2 & 2 & 2 & 2 & 4.46 & 0 & 0 &  & \B{1}{2} & 4.85 & 2 & 0 &  & \B{2}{3} \\
4.8  & 0 & 0 & & \B{1}{2} & 4.47* & 1 & 2 & & 2 & 4.86 & 2 & 0 & & \B{2}{3} \\
4.9  & 1 & 1 & & \B{1}{3} & 4.48 & 3 & 1 & 3 & 3 & 4.87 & 4 & 1 & & \B{4}{5} \\
4.10 & 1 & 1 & & \B{1}{2} & 4.49 & 1 & 1 & 1 & 1 & 4.88 & 1 & 1 & 1 & 1 \\
4.11 & 2 & 1 & 2 & 2 & 4.50 & 1 & 1 & 1 & 1 & 4.89 & 4 & 0 & 4 & 4 \\
4.12 & 0 & 0 & & \B{1}{2} & 4.51 & 0 & 0 &  & \B{1}{2} & 4.90 & 0 & 0 & & \B{1}{2} \\
4.13 & 0 & 0 & & \B{1}{2} & 4.52 & 1 & 1 & 1 & 1 & 4.91 & 4 & 2 & & \B{4}{6}  \\
4.14 & 1 & 1 & & \B{1}{2} & 4.53 & 2 & 2 & 2 & 2 & 4.92 & 3 & 1 & & \B{3}{5} \\
4.15 & 2 & 1 & 2 & 2 & 4.54 & 1 & 1 & 1 & 1 & 4.93 & 2 & 1 & 2 & 2  \\
4.16 & 0 & 0 & & \B{1}{2} & 4.55* & 0 & 0 & & 1 & 4.94 & 1 & 1 & & \B{1}{6}\\
4.17 & 1 & 1 & 1 & 1 & 4.56* & 0 & 0 &  & 1 & 4.95 & 3 & 1 & & \B{3}{5} \\
4.18 & 1 & 1 & 1 & 1 & 4.57 & 1 & 1 & 1 & 1 & 4.96 & 2 & 0 & & \B{2}{3}\\
4.19 & 1 & 1 & & \B{1}{2} & 4.58* & 0 & 0 & & 1 & 4.97 & 1 & 1 & & \B{1}{2} \\
4.20 & 1 & 1 & 1 & 1 & 4.59* & 0 & 0 & & 1 & 4.98 & 0 & 0 & & \B{1}{3} \\
4.21 & 2 & 0 & 2 & 2 & 4.60 & 1 & 1 & 1 & 1 & 4.99 & 0 & 0 & & \B{1}{3} \\
4.22 & 1 & 1 & 1 & 1 & 4.61 & 1 & 1 &  & \B{1}{5} & 4.100 & 2 & 2 & & \B{2}{7} \\
4.23 & 1 & 1 & 1 & 1 & 4.62 & 3 & 1 & & \B{1}{4} & 4.101 & 3 & 1 & & \B{3}{4} \\
4.24 & 2 & 1 &  & \B{2}{3} & 4.63 & 2 & 1 & 2 & 2 & 4.102 & 2 & 2 & & \B{2}{4} \\
4.25 & 2 & 2 & 2 & 2 & 4.64 & 1 & 1 & 1 & 1 & 4.103 & 3 & 1 & & \B{3}{6} \\
4.26* & 1 & 2 & & 2 & 4.65 & 2 & 0 & & \B{2}{4} & 4.104 & 3 & 1 & & \B{3}{6} \\
4.27 & 1 & 1 & & \B{1}{2} & 4.66 & 2 & 1 & & \B{2}{3} & 4.105 & 0 & 0 & & \B{1}{5} \\
4.28 & 2 & 2 & 2 & 2 & 4.67 & 1 & 1 & & \B{1}{3} & 4.106 & 2 & 0 & & \B{2}{4} \\
4.29 & 2 & 1 & 2 & 2 & 4.68 & 0 & 0 & & \B{1}{2} & 4.107 & 0 & 0 & & \B{1}{4} \\
4.30 & 1 & 1 & & \B{1}{2} & 4.69 & 1 & 1 & & \B{1}{3} & 4.108 & 0 & 0 & & \B{1}{4} \\
\hline
\end{tabular}
\end{center}
\caption{Table of forbidden numbers for virtual knots with at most four crossings.} \label{Ta:forbidden}

\end{table}


\end{document}